\documentclass[12pt]{article}
\usepackage{amsmath}
\usepackage{amssymb}
\usepackage{mathrsfs}
\usepackage{color}
\usepackage{array}
\numberwithin{equation}{section}
\usepackage{datetime}
\usepackage[linkcolor=black,urlcolor=black,colorlinks=true]{hyperref}
\usepackage{enumerate}
\def \beq {\begin{equation}}
\def \eeq {\end{equation}}
\def \ba {\begin{array}}
\def \ea {\end{array}}
\def \dis {\displaystyle}

\renewcommand{\r}{\mathop{\rightarrow}}
\def\r{\rightarrow}
\renewcommand{\r}{\mathop{\rightarrow}}
\newcommand{\fdem}{\hspace*{\fill}~$\Box$\par\endtrivlist\unskip}

\def \div {\mbox{\rm div}\,}

\def \de {\delta}

\def \ep {\varepsilon}

\def \ph {\varphi}

\newcommand{\N}{\mathbb{N}}     
\newcommand{\Z}{\mathbb{Z}}
\newcommand{\R}{\mathbb{R}}

\def \cD {\mathscr{D}}

\def \cI {\mathscr{I}}

\def \cM {\mathscr{M}}

\def \beq {\begin{equation}}
\def \eeq {\end{equation}}
\def \ba {\begin{array}}
\def \ea {\end{array}}
\def \bs {\bigskip}

\def \ecart {\noalign{\medskip}}

\newenvironment{proof}[1]{\textit{Proof#1.\,}}{\fdem}

\newtheorem{atheo}{Theorem}[section]
\newtheorem{acor}{Corollary}[section]
\newtheorem{alem}{Lemma}[section]
\newtheorem{arem}{Remark}[section]
\newtheorem{Aexa}{Exemple}[section]

\newtheorem{apro}[alem]{Proposition}
\newtheorem{adef}[alem]{Definition}
\title{Revisiting the asymptotics of the flow for some dynamical systems on the torus}
\author{\normalsize Marc Briane \& Lo\"\i c Herv\'e
\\
\normalsize Univ Rennes, INSA Rennes,  CNRS, IRMAR - UMR 6625, F-35000 Rennes, France
\\
\normalsize mbriane@insa-rennes.fr \& loic.herve@insa-rennes.fr
}
\begin{document}
\maketitle
\vskip .5cm\noindent
\begin{abstract}
In this paper we study the large time asymptotics of the flow of a dynamical system $X'=b(X)$ posed in the $d$-dimensional torus.
Rather than using the classical unique ergodicity condition which is not fulfilled if $b$ vanishes at different points, we only assume that the set of the averages of $b$ with respect to the invariant probability measures for the flow is reduced to a singleton. We also rewrite the Liouville theorem which holds for any invariant probability measure $\mu$, namely $\mu\,b$ is divergence free, as a divergence-curl formula satisfied by any regular periodic function. The combination of these two tools turns out to be a new approach to get the asymptotics for some flows. This allows us to obtain the desired asymptotics in any dimension when $b = a\,\xi$ with $a$ a possibly vanishing periodic nonnegative function and $\xi$ a nonzero vector in $\R^d$, or when $b = A\nabla v$ with $A$ a periodic nonnegative symmetric matrix-valued function and $v$ a periodic function.
\end{abstract}
\par\bs\noindent
%%%%%%%%%%
\section{Introduction}
This paper is devoted to the large time asymptotics of the solution $X(t,x)$ to the dynamical system
\beq\label{bXi}
\left\{\ba{ll}
\dis {\partial X\over\partial t}(t,x)=b(X(t,x)), & t\in\R
\\ \ecart
X(0,x)=x\in\R^d,
\ea\right.
\eeq
where $b$ is a $C^1$-regular $\Z^d$-periodic vector field in $\R^d$. More precisely, we focus on the existence of the limit of $X(t,x)/t$ as $t\to\infty$ for $x\in\R^d$. This problem is strongly connected to the asymptotic behavior of the transport equation with an oscillating velocity
\beq\label{traequ}
{\partial u_\ep\over\partial t}(t,x)+b\left({x\over\ep}\right)\cdot\nabla_x u_\ep(t,x)=0\quad\mbox{for }(t,x)\in[0,\infty)\times\R^d,
\eeq
which is dealt with in \cite{Bre,HoXi,Gol1,Gol2,Tas,Bri2}.
Otherwise, this question naturally arises in ergodic theory since it involves the {\em flow} $T_t$ defined by
\beq\label{Tti}
(T_t\ph)(x):=\ph(X(t,x))\quad\mbox{for $\ph$ continuous and $\Z^d$-periodic in }\R^d,\ x\in\R^d,
\eeq
which is associated with system \eqref{bXi}, and the Borel measures $\mu$ on the torus $Y_d:=\R^d/\Z^d$ which are {\em invariant for the flow},
{\em i.e.}
\beq\label{muinvTt}
\forall\,t\in\R,\quad \mu\circ T_t=\mu.
\eeq
A strengthened variant of the famous Birkhoff ergodic theorem \cite[Theorem~2, Section~1.8]{CFS} claims that if the flow is {\em uniquely ergodic}, {\em i.e.} there exists a unique probability measure $\mu$ on $Y_d$ invariant for the flow, then any continuous $\Z^d$-periodic function $f$ satisfies
\beq\label{Ttuninv}
\forall\,x\in\R^d,\quad \lim_{t\to\infty}\left[{1\over t}\int_0^t f(X(s,x))\,ds\right]=\int_{Y_d}f(y)\,d\mu(y),
\eeq
and the converse actually holds true.
In the particular case where $f:=b$, limit \eqref{Ttuninv} yields
\beq\label{asyXi}
\forall\,x\in\R^d,\quad \lim_{t\to\infty}{X(t,x)\over t}=\int_{Y_d}b(y)\,d\mu(y).
\eeq
The unique ergodicity condition seems to be a rather restrictive condition on the flows of type~\eqref{bXi}. For example, if the vector field $b$ vanishes at two points $x\neq y$ in the torus, then any convex combination of the Dirac masses $\de_x$ and $\delta_y$ are invariant for the flow. When $b$ does not vanish in $\R^d$, the unique ergodicity of the flow is not clear. This is true in dimension one (see Proposition~\ref{pro.1D}) but the situation is more complicated in higher dimension. As best we know the asymptotics of the flow is completely known only in dimension two for a non-vanishing vector field $b$. Specifically in dimension two, Peirone~\cite{Pei} obtained the asymptotics of the flow through the following alternative:
\begin{itemize}
\item Under the existence of a periodic trajectory $X(\cdot,y)$ of \eqref{bXi} with respect to the torus (see \eqref{XtperY}), the limit of \eqref{asyXi} does exist but may depend on $x$.
\item In the absence of periodic trajectory with respect to the torus, the limit of \eqref{asyXi} exists independently of $x$. To this end, Peirone used an one-dimensional ergodicity result transversally to Siegel's curve, but his approach has nothing to do with the unique ergodicity of the flow.
\end{itemize}
On the other hand, when the two-dimensional flow associated with a non-vanishing vector field $b=(b_1,b_2)$ admits an invariant measure with a positive $\Z^2$-periodic density $\sigma$ with respect to the Lebesgue measure, or equivalently, by virtue of Liouville's theorem $\sigma b$ is divergence free in~$\R^2$, then Kolmogorov's theorem \cite{Kol} implies the existence of a diffeomorphism on the torus which transforms the dynamical system \eqref{bXi} into the dynamical system
\beq\label{aYi}
\left\{\ba{ll}
\dis {\partial Y\over\partial t}(t,y)=a(Y(t,y))\,\xi, & t\in\R
\\ \ecart
Y(0,x)=y\in\R^2,
\ea\right.
\eeq
where $a$ is a regular $\Z^2$-periodic vector field in $\R^2$ and $\xi=(\xi_1,\xi_2)$ is a nonzero constant vector in~$\R^2$.
In this two-dimensional setting with the additional assumption that $b_1$ is nonvanishing, Tassa \cite{Tas} also obtained the asymptotics of the flow through the following alternative which is strongly connected to the above Peirone's alternative:
\begin{itemize}
\item If $\xi_1$ and $\xi_2$ are rationally dependent, the limit of \eqref{asyXi} exists but does depend on $x$.
\item  If $\xi_1$ and $\xi_2$ are rationally independent, or equivalently the so-called {\em rotation number} is irrational (see, {\em e.g.}, \cite[Chapter~I, Section~4.1]{Sin}), the limit of \eqref{asyXi} exists independently of~$x$. The more general result where the vector field $b$ (rather than $b_1$) is nonvanishing with an invariant density $\sigma\in C^5(Y_2)$ is due to Kolmogorov~\cite{Kol}.
\end{itemize}
Note that in \cite{Pei,Tas} the nonvanishing condition of the vector field $b$ is essential to obtain the asymptotics of the flow.
\par
Actually, the unique ergodicity of the flow \eqref{Tti} associated with $b$ is not needed to get the desired asymptotics \eqref{asyXi} which is much less restrictive than \eqref{Ttuninv}. Indeed, our approach consists in replacing the unique ergodicity assumption by the weaker uniqueness condition
\beq\label{Cbi}
\#\left\{\int_{Y_d}b(y)\,d\mu(y):\mu\mbox{ is an invariant probability measure for the flow}\right\}=1.
\eeq
This new condition is obtained by observing that the Birkhoff time average in \eqref{asyXi} is only addressed (contrary to \eqref{Ttuninv}) to the function $b$. Then, we revisit (see Proposition~\ref{pro.invmeas}) the proof of the existence of an invariant probability measure for the flow defined on a compact space (the torus here).
We conclude by proving (see Theorem~\ref{thm.Cb}) that condition \eqref{Cbi} turns out to be equivalent to asymptotics \eqref{asyXi}.
On the other hand, we show (see Theorem~\ref{thm.divcurl} and Remark~\ref{rem.divcurl}) that the Liouville theorem satisfied by an invariant measure $\mu$ for the flow associated with $b$ can be regarded as a divergence-curl result such that for any $\Z^d$-periodic $C^1$-function $\psi$,
\beq\label{divcurli}
\int_{Y_d} b(y)\cdot\nabla\psi(y) \,d\mu(y)=0.
\eeq
Our new approach consists in combining the tools \eqref{Cbi} and \eqref{divcurli} to obtain the limit of $X(t,x)/t$ as $t\to\infty$ for some dynamical systems \eqref{bXi}. So, we get the desired asymptotics for the dynamical system \eqref{aYi} (see Proposition~\ref{pro.asyxi}) assuming that the $\Z^d$-periodic function $a$ is nonnegative and may vanish contrary to \cite{Pei,Tas}, and that the vector $\xi$ satisfies one of the two following conditions which extend in any dimension the above two-dimensional conditions obtained by Peirone (see Remark~\ref{rem.P} and Proposition~\ref{pro.P}) and Tassa:
\begin{itemize}
\item There exists $T>0$ such that $T\xi\in\Z^d$, which leads us to a limit of $X(t,x)/ t$ depending on~$x$.
\item For any $k\in\Z^d\setminus\{0\}$, $\xi\cdot k\neq 0$, which leads us to the limit
\[
\forall\,x\in\R^d,\quad \lim_{t\to\infty}{X(t,x)\over t}=a^*\xi\quad\mbox{with}\quad
a^*:=\left\{\ba{cl}
\dis \left(\int_{Y_d}a^{-1}(y)\,dy\right)^{-1} & \mbox{if $a>0$ in }\R^d
\\ \ecart
0 & \mbox{if $a$ vanishes in }\R^d.
\ea\right.
\]
\end{itemize}
When $a$ vanishes at different points in the torus, the unique ergodicity of the flow \eqref{Ttuninv} is not satisfied, while the asymptotics of the flow holds true for $b$ in \eqref{asyXi}.
The former asymptotics result easily extends (see Corollary~\ref{cor.bPhaxi}) to the case where the vector field $b$ is rectifiable to a fixed direction $\xi$ through a diffeomorphism $\Phi$ on the torus, {\em i.e.} $b=a\circ\Phi\,\nabla\Phi^{-1}\xi$ for some $\Z^d$-periodic nonnegative function $a$.
Finally, again using \eqref{Cbi} and \eqref{divcurli} we obtain the zero vector limit in \eqref{asyXi} when $b$ is a current field (see Proposition~\ref{pro.ADv}), {\em i.e.} $b=A\nabla v$ with $A$ a regular $\Z^d$-periodic nonnegative symmetric matrix-valued conductivity and $v$ a regular $\Z^d$-periodic potential.
\subsection*{Notation}
\begin{itemize}
\item $Y_d$ for $d\geq 1$, denotes the $d$-dimensional torus $\R^d/\Z^d$, which is identified to the cube $[0,1)^d$ in $\R^d$.
\item $C^k_c(\R^d)$ for $k\in\N\cup\{\infty\}$, denotes the space of the real-valued functions in $C^k(\R^d)$ with compact support.
\item $C^k_\sharp(Y_d)$ for $k\in\N\cup\{\infty\}$, denotes the space of the real-valued functions $f\in C^k(\R^d)$ which are $\Z^d$-periodic, {\em i.e.}
\[
\forall\,k\in\Z^d,\ \forall\,x\in\R^d,\quad f(x+k)=f(x).
\]
\item $L^p_\sharp(Y_d)$ for $p\geq 1$, denotes the space of the real-valued functions in $L^p_{\rm loc}(\R^d)$ which are $\Z^d$-periodic.
\item $\cM(Y_d)$ denotes the space of the Radon measures on $Y_d$, and $\cM_p(Y_d)$ denotes the space of the probability measures on $Y_d$.
\item $\cD'(\R^d)$ denotes the space of the distributions on $\R^d$.
\end{itemize}
\subsection*{Definitions and recalls}
Let $b:\R^d\to\R^d$ be a vector-valued function in $C^1_\sharp(Y_d)^d$.
Consider the dynamical system
\beq\label{bX}
\left\{\ba{ll}
\dis {\partial X\over\partial t}(t,x)=b(X(t,x)), & t\in\R
\\ \ecart
X(0,x)=x\in\R^d.
\ea\right.
\eeq
The solution $X(\cdot,x)$ of \eqref{bX} which is known to be unique (see, {\em e.g.}, \cite[Section~17.4]{HSD}) is associated with the flow $(T_t)_{t\in\R}$, defined by
\beq\label{Tt}
T_t(\ph)(x):=\ph\big(X(t,x)\big)\quad\mbox{for }\ph\in C^0_\sharp(Y_d)\mbox{ and }x\in\R^d,
\eeq
which satisfies the semi-group property
\beq\label{sgroup}
\forall\,s,t\in\R,\quad T_{s+t}=T_s\circ T_t,
\eeq
and is well defined in the torus since
\beq\label{XxperY}
\forall\,t\in\R,\ \forall\,x\in\R^d,\ \forall\,k\in\Z^d,\ \quad X(t,x+k)=X(t,x)+k.
\eeq
Property~\eqref{XxperY} follows immediately from the uniqueness of the solution $X$ to \eqref{bX} combined with the $\Z^d$-periodicity of $b$.
\par
A solution $X(\cdot,x)$ to \eqref{bX} is said to be {\em periodic in the torus} if there exists $\tau>0$ and $k\in\Z^d$ such that
\beq\label{XtperY}
\forall\,t\in\R,\quad X(t+\tau,x)=X(t,x)+k.
\eeq
If $k=0$ the solution is said to be periodic in $\R^d$.
\par
A measure $\mu$ in $\cM_p(Y_d)$ is said to be {\em invariant for the flow} $T_t$ (see, {\em e.g.}, \cite[Chap. 2]{CFS}) if
\beq\label{invmu}
\forall\,t\in\R,\ \forall\,\psi\in C^0_\sharp(Y_d),\quad \int_{Y_d}T_t(\psi)(y)\,d\mu(y)=\int_{Y_d}\psi\big(X(t,y)\big)\,d\mu(y)=\int_{Y_d}\psi(y)\,d\mu(y),
\eeq
or equivalently, the image measure of $\mu$ by the flow $T_t$ agrees with $\mu$.
%%%%%%%%%%
\section{Some general results}\label{s.genres}
\subsection{Existence of invariant probability measures for the flow}
When a flow preserves the set of the continuous functions on a compact metric space, the existence of an invariant probability measure  for the flow is a classical statement which can be derived thanks to a weak compactness argument applied to sequences of probability measures defined from the Birkhoff time averages in (1.5) (see, {\em e.g.}, \cite[Theorem~1, Section~1.8]{CFS} in the discrete time case). The following result adapts this statement restricting it to the limit points of the Birkhoff time averages for a given function.
\begin{apro}\label{pro.invmeas}
Let $b\in C^1_\sharp(Y_d)^d$. There exists an invariant {probability measure on $Y_d$} for the flow $T_t$ \eqref{Tt} associated with $b$.
Moreover, let $g\in C^0_\sharp(Y_d)$ and $x\in\R^d$ be fixed, and let $(t_n)_{n\in\N}\in\R^\N$ be such that $\lim_n t_n = \infty$. Then, for any limit point $a$ of the sequence $(u_n)_{n\in\N}\in\R^\N$ defined by
\beq\label{ung}
u_n := \frac{1}{t_n}\int_0^{t_n} g(X(s,x))\,ds,\quad n\in\N,
\eeq
there exists a probability measure $\mu\in \cM_p(Y_d)$ (depending {\em a priori} on $x$ and $g$) which is invariant for the flow $T_t$ and satisfies
\beq\label{afmu}
a=\int_{Y_d} g(y)\, d\mu(y).
\eeq
\end{apro}
\begin{proof}{}
As above mentioned the following lemma is classical in ergodic theory. For the reader's convenience its proof is postponed to the Appendix. 
\begin{alem} \label{lem-extrac}
Let $x\in\R^d$, let $(r_n)_{n\in\N}\in\R^\N$ be such that $\lim_n r_n = \infty$, and let $\nu_n\in \cM_p(Y_d)$, $n\in\N$, be the probability measure defined by
\beq\label{nun}
\int_{Y_d} f(y)\, d\nu_n(y) = \frac{1}{r_n} \int_0^{r_n} f(X(s,x))\,ds\quad\mbox{for }f\in C^0_\sharp(Y_d).
\eeq
Then, there exists a subsequence $(\nu_{n_k})_{k\in\N}$ of $(\nu_n)_{n\in\N}$ which converges weakly~$*$ to some probability measure $\mu\in \cM_p(Y_d)$ which is invariant for the flow $T_t$. 
\end{alem}
\par
Let $a$ be a limit point of the sequence $(u_n)_{n\in\N}$ \eqref{ung}, namely 
\[
a = \lim_{n\to\infty} \frac{1}{t_{\theta(n)}}\int_0^{t_{\theta(n)}} g(X(t,x))\,dt,
\]
for some subsequence $(t_{\theta(n)})_{n\in\N}$ of $(t_n)_{n\in\N}$.
Set $r_n := t_{\theta(n)}$, and consider the associated sequence $(\nu_n)_{n\in\N}$ of probability measures on $Y_d$ given by \eqref{nun}. 
By Lemma~\ref{lem-extrac} we can extract a subsequence $(\nu_{n_k})_{k\in\N}$ which converges weakly~$*$ to some invariant measure $\mu\in \cM_p(Y_d)$ for the flow~$T_t$.
We thus have
\[
\forall\,f\in C^0_\sharp(Y_d),\quad \lim_{k\to\infty} \int_{Y_d} f(y)\, d\nu_{n_k}(y) = \int_{Y_d} f(y)\, d\mu(y),
\]
which implies in particular that
\[
a = \lim_{k\to\infty}\frac{1}{r_{n_k}}\int_0^{r_{n_k}} g(X(s,x))\,ds
= \lim_{k\to\infty} \int_{Y_d} g(y)\, d\nu_{n_k}(y) = \int_{Y_d} g(y)\, d\mu(y).
\]
%The proof is now complete. 
\end{proof}
\begin{arem}
The second assertion of Proposition~\ref{pro.invmeas} may provide a simple way to find the limit of Birkhoff's time averages for a specific function, when the general form of the limit in Birkhoff's theorem is unknown. For instance, if the support of any invariant probability measure for the flow is contained in a subset $F$ of $Y_d$ and if $g\in C^0_\sharp(Y_d)$ vanishes on $F$, then
\[
\forall\,x\in Y_d,\quad\lim_{t\r \infty}\left[{1\over t}\int_0^{t} g(X(s,x))\,ds\right]=0,
\]
since any limit point of these averages  is zero by \eqref{afmu}.
\end{arem}
%%%%%%%%%%%%%%%%%%%%%%%%%%
\subsection{A criterium for the asymptotics of the flow}
It is known (see, {\em e.g.}, \cite[Theorem~2, Section~1.8]{CFS} for the discrete case) that the uniqueness of an invariant measure $\mu\in \cM_p(Y_d)$ for the flow $T_t$ is equivalent to the pointwise property:
\beq\label{xfX}
\forall\,x\in\R^d,\ \forall\,f\in C^0_\sharp(Y_d),\quad \lim_{t\to\infty}\left({1\over t}\int_0^t f(X(s,x))\,ds\right)=\int_{Y_d}f(y)\,d\mu(y).
\eeq
The following result which is new as the best we know, allows us to restrict condition~\eqref{xfX} to $f=b$ and to derive the asymptotics of the flow at each point in $\R^d$, by assuming the uniqueness of the averages of $b$ with respect to the invariant probability measures on $Y_d$ for the flow rather than the uniqueness of an invariant measure for the flow.
\begin{atheo}\label{thm.Cb}
Let $b\in C^1_\sharp(Y_d)^d$. Define the sets
\beq\label{IbCb}
\cI_b:=\big\{\mu\in\cM_p(Y_d): \mu\mbox{ is invariant for the flow }T_t\big\}\quad\mbox{and}\quad
C_b:=\left\{\int_{Y_d}b\,d\mu:\mu\in \cI_b\right\}.
\eeq
Then, $\cI_b$ is a nonempty set, and $C_b$ is a nonempty compact convex set of $\R^d$.
\par\noindent
Moreover, the following equivalence holds for any $\zeta\in\R^d$,
\beq\label{Cbsin}
C_b=\{\zeta\} \quad\Leftrightarrow\quad \forall\,x\in\R^d,\;\;\lim_{t\to\infty}{X(t,x)\over t}=\zeta.
\eeq
\end{atheo}
\begin{proof}{}
By virtue of Proposition~\ref{pro.invmeas} the sets $\cI_b$ and $C_b$ are nonempty. The set $C_b$ is a convex set in $\R^d$ since $\cI_b$ is clearly convex. Using the compactness of $\cM_p(Y_d)$ for the weak~$*$ topology, we get that $C_b$ is a closed set in $\R^d$. We also have $C_b\subset[0,\|b\|_\infty]$, so that $C_b$ is a compact set of~$\R^d$.
\par
Now, assume that $C_b=\{\zeta\}$.
Let  $x\in \R^d$, let $(t_n)_{n\in\N}\in\R^\N$ be such that $\lim_n t_n = \infty$, and define the sequence $(u_n)_{n\in\N}$ by \eqref{ung} with the function $g:=b\cdot\xi$ for $\xi\in\R^d$.
Let $a$ be a limit point of the sequence $(u_n)_{n\in\N}$.
By Proposition~\ref{pro.invmeas} there exists an invariant measure $\mu\in \cM_p(Y_d)$ for the flow $T_t$ satisfying
\[
a=\int_{Y_d} b(y)\cdot\xi\, d\mu(y),
\]
which by hypothesis implies that $a=\zeta\cdot\xi$. Hence, $\zeta\cdot\xi$ is the unique limit point of the bounded sequence $(u_n)_{n\in\N}$ which thus converges to $\zeta\cdot\xi$. Therefore, due to the arbitrariness of the sequence $(t_n)_{n\in\N}$ we obtain that
\[
\forall\,x\in\R^d,\quad \lim_{t\to\infty}{X(t,x)\over t}=\lim_{t\to\infty}\left({1\over t}\int_0^t b(X(s,x))\,ds\right)=\zeta.
\]
\par
Conversely, assume that the right-hand side of \eqref{Cbsin} holds, which implies that
\[
\forall\,x\in\R^d,\quad \lim_{t\to\infty}\left({1\over t}\int_0^t b(X(s,x))\,ds\right)=\zeta.
\]
Then, integrating over $Y_d$ the former equality with respect to any probability measure $\mu\in\cI_b$, then applying successively Lebesgue's dominated convergence theorem and Fubini's theorem, we get that
\[
\ba{ll}
\zeta & \dis =\lim_{t\to\infty}\int_{Y_d}\left({1\over t}\int_0^t b(X(s,x))\,ds\right)d\mu(x)
\\ \ecart
& \dis =\lim_{t\to\infty}{1\over t}\int_0^t\left(\int_{Y_d} b(X(s,x))\,d\mu(x)\right)ds
=\int_{Y_d} b(x)\,d\mu(x).
\ea
\]
which shows that $C_b=\{\zeta\}$. This concludes the proof of \eqref{Cbsin}.
\end{proof}
\subsection{Liouville's theorem and a divergence-curl result}
Liouville's theorem provides a criterium for a probability measure on a smooth compact manifold in $\R^d$ (see, {\em e.g.}, \cite[Theorem 1, Section~2.2]{CFS}) to be invariant for the flow.
The next result revisits this theorem in $\cM_p(Y_d)$ in association with a divergence-curl result on the torus.
\begin{atheo}\label{thm.divcurl}
Let $b\in C^1_\sharp(Y_d)^d$ and let $\mu\in\cM_p(Y_d)$. We define the Radon measure $\tilde{\mu}\in \cM(\R^d)$ on~$\R^d$ by
\beq\label{tmu}
\int_{\R^d}\ph(x)\,d\tilde{\mu}(x):=\int_{Y_d} \ph_\sharp(y)\,d\mu(y)\quad\mbox{where}\quad \ph_\sharp(\cdot):=\sum_{k\in\Z^d}\ph(\cdot+k)\quad\mbox{for }\ph\in C^0_c(\R^d).
\eeq
Then, $\mu$ is invariant for the flow $T_t$ if, and only if, one of the two following conditions is satisfied:
\beq\label{dtmub=0}
\div(\tilde{\mu}\,b)=0\quad\mbox{in }\cD'(\R^d),
\eeq
\beq\label{dmub=0}
\forall\,\psi\in C^1_\sharp(Y_d),\quad \int_{Y_d} b(y)\cdot\nabla\psi(y)\,d\mu(y)=0.
\eeq
\end{atheo}
\begin{proof}{}
Assume that $\mu$ is invariant for the flow, {\em i.e.} \eqref{invmu}.
Let $\ph\in C^1_c(\R^d)$. Since by \eqref{XxperY} we have for any $t\in\R$ and $y\in \R^d$,
\beq\label{phdph}
\big[\ph(X(t,\cdot))\big]_\sharp(y)=\sum_{k\in\Z^d}\ph(X(t,y+k))=\sum_{k\in\Z^d}\ph(X(t,y)+k)=\ph_\sharp(X(t,y)),
\eeq
it follows from \eqref{tmu} and the invariance of $\mu$ that
\[
\ba{ll}
\forall\,t\in\R, & \dis \int_{\R^d} \ph(X(t,x))\,d\tilde{\mu}(x)=\int_{Y_d} \big[\ph(X(t,\cdot))\big]_\sharp(y)\,d\mu(y)=\int_{Y_d} \ph_\sharp(X(t,y))\,d\mu(y)=
\\*[1.em]
& \dis \int_{Y_d} \ph_\sharp(y)\,d\mu(y)=\int_{\R^d} \ph(x)\,d\tilde{\mu}(x).
\ea
\]
Taking the derivative of the former expression with respect to $t$, we get that
\[
\forall\,t\in\R,\quad \int_{\R^d} b(X(t,x))\cdot\nabla\ph(X(t,x))\,d\tilde{\mu}(x)=0,
\]
which at $t=0$ yields
\beq\label{dtmub=02}
\forall\,\ph\in C^1_c(\R^d),\quad \int_{\R^d}b(x)\cdot\nabla\ph(x)\,d\tilde{\mu}(x)=0,
\eeq
namely the variational formulation of the distributional equation \eqref{dtmub=0}.
\par
Conversely, assume that equation \eqref{dtmub=0} holds true. Let $\ph\in C^1_c(\R^d)$ and define the function $\phi\in C^1(\R\times\R^d)$ by $\phi(t,x):=\ph(X(t,x))$.
By the semi-group property \eqref{sgroup} we have for any $s,t\in\R$ and $x\in\R^d$,
\[
\ba{ll}
\dis {\partial\over\partial s}\big(\phi(s+t,X(-s,x))\big) & \dis ={\partial\over\partial s}\big(\phi(t,x)\big)=0
\\ \ecart
& \dis ={\partial\phi\over\partial s}(s+t,X(-s,x))-b(X(-s,x))\cdot\nabla_x\phi(s+t,X(-s,x)),
\ea
\]
which at $s=0$ gives the classical transport equation
\beq\label{phib}
\forall\,t\in\R,\ \forall\,x\in\R^d,\quad{\partial\phi\over\partial t}(t,x)=b(x)\cdot\nabla_x\phi(t,x).
\eeq
Hence, since $\ph(X(t,\cdot))$ is in $C^1(\R^d)$ and has a compact support independent of $t$ when $t$ lies in a compact set of $\R$, we deduce from \eqref{phib} and \eqref{dtmub=0} that
\[
\forall\,t\in\R,\quad{d\over\,dt}\left(\int_{\R^d} \ph(X(t,x))\,d\tilde{\mu}(x)\right)=\int_{\R^d} b(x)\cdot\nabla_x\big(\ph(X(t,x))\big)\,d\tilde{\mu}(x)=0,
\]
or equivalently,
\[
\forall\,t\in\R,\quad \int_{\R^d} \ph(X(t,x))\,d\tilde{\mu}(x)=\int_{\R^d} \ph(x)\,d\tilde{\mu}(x).
\]
On the other hand, we have the following result.
\begin{alem}[\cite{Bri1}, Lemma~3.5]\label{lem.phd}
For any smooth $\Z^d$-periodic function $\psi\in C^\infty_\sharp(Y_d)$, there exists a smooth function with compact support $\ph\in C^\infty_c(\R^d)$ such that $\psi=\ph_\sharp$.
\end{alem}
Hence, using relation \eqref{phdph} and definition \eqref{tmu} we get that for any $\psi\in C^\infty_\sharp(Y)$,
\[
\ba{ll}
\forall\,t\in\R, & \dis \int_{Y_d} \psi(X(t,y))\,d\mu(y)=\int_{Y_d} \ph_\sharp(X(t,y))\,d\mu(y)=\int_{\R^d} \ph(X(t,x))\,d\tilde{\mu}(x)=
\\*[1.em]
&  \dis \int_{\R^d} \ph(x)\,d\tilde{\mu}(x)=\int_{Y_d} \psi(y)\,d\mu(y),
\ea
\]
which shows that $\mu$ is invariant for the flow. We have just proved the equivalence between the invariance of $\mu$ for the flow and the distributional equation \eqref{dtmub=0} satisfied by $\tilde{\mu}$.
\par
Finally, the equivalence between \eqref{dtmub=0}, or equivalently \eqref{dtmub=02}, and \eqref{dmub=0} is a straightforward consequence of the relation
\[
\forall\,\ph\in C^1_c(\R^d),\quad \int_{\R^d}b(x)\cdot\nabla\ph(x)\,d\tilde{\mu}(x)=\int_{Y_d} b(y)\cdot\nabla\ph_\sharp(y)\,d\mu(y)
\]
which itself follows from $[b\cdot\nabla\ph]_\sharp=b\cdot\nabla\ph_\sharp$ and \eqref{tmu}, combined with Lemma~\ref{lem.phd}.
\end{proof}
\begin{arem}\label{rem.divcurl}
Equation \eqref{dmub=0} can be considered as the divergence free of the vector-valued measure $\mu\,b$ in the torus $Y_d$, while equation \eqref{dtmub=0} is exactly the divergence free of the vector-valued measure $\tilde{\mu}\,b$ in the space $\R^d$. Equation \eqref{dmub=0} is also equivalent to
\beq\label{divcurl}
\forall\,\nabla\psi\in C^0_\sharp(Y_d)^d,\quad \int_{Y_d} b(y)\cdot\nabla\psi(y) \,d\mu(y)=\left(\int_{Y_d} b(y)\,d\mu(y)\right)\cdot
\left(\int_{Y_d} \nabla\psi(y)\,dy\right),
\eeq
since
\[
\nabla\psi\in C^0_\sharp(Y_d)^d\;\Leftrightarrow\; \left(x\mapsto \psi(x)-x\cdot{ \int_{Y_d} \nabla\psi(y)\,dy}\right)\in C^1_\sharp(Y_d).
\]
Condition \eqref{divcurl} has to be regarded as a divergence-curl result combining the invariant measure $\mu$ for the divergence free vector field $\mu\,b$ and the Lebesgue measure for the gradient field $\nabla\psi$.
\end{arem}
%%%%%%%%%%
\section{Application to the asymptotics of the flow}
First of all, we apply the tools of Section~\ref{s.genres} to the one-dimensional case.
\subsection{The one-dimensional case}
We have the following result.
\begin{apro}\label{pro.1D}
Let $b\in C^1_\sharp(Y_1)$. We have the following alternative:
\begin{enumerate}[$(i)$]
\item If $b\neq 0$ sur $Y_1$, then $\dis {\underline{b}/b(y)}\,dy$ is the unique invariant measure for the flow $T_t$ associated with $b$, and
\beq\label{asyXb­0}
\forall\,x\in\R,\quad \lim_{t\to\infty} \frac{X(t,x)}{t} = \underline{b}:=\left(\int_{Y_1}{dy\over b(y)}\right)^{-1}.
\eeq
\item If $b$ does vanish in $Y_1$, then 
\beq\label{asyXb=0}
\forall\,x\in\R,\quad \lim_{t\to\infty} \frac{X(t,x)}{t} = 0.
\eeq
\end{enumerate}
\end{apro}
\begin{proof}{}
\par\noindent
{\it Case $(i)$.} Let $\mu\in\cM_p(Y_1)$ be an invariant measure for the flow $T_t$.
By the condition \eqref{dtmub=0} of Theorem~\ref{thm.divcurl} there exists a constant $C\in\R$ such that $\tilde{\mu}\,b=C$ in $\R$, or equivalently,
\[
\forall\,\ph\in C^0_c(\R),\quad \int_{\R}\ph(x)\,b(x)\,d\tilde{\mu}(x)=C\int_{\R}\ph(x)\,dx.
\]
Hence, by the definition \eqref{tmu} of $\tilde{\mu}$ we have
\[
\forall\,\ph\in C^0_c(\R),\quad \int_{Y_1}\ph_\sharp(y)\,b(y)\,d\mu(y)=\int_{\R}\ph(x)\,b(x)\,d\tilde{\mu}(x)
=C\int_{\R}\ph(x)\,dx=C\int_{Y_1}\ph_\sharp(y)\,dy.
\]
Therefore, from Lemma~\ref{lem.phd} we deduce that
\beq\label{psiC}
\forall\,\psi\in C^0_\sharp(Y_1),\quad \int_{Y_1}\psi(y)\,b(y)\,d\mu(y)=C\int_{Y_1}\psi(y)\,dy.
\eeq
Taking $\psi=1/b$ in \eqref{psiC} we get that $C=\underline{b}$, which implies that $\mu(dy)={\underline{b}/b(y)}\,dy$ is thus the unique invariant measure for the flow.
\par
As a consequence the set $\cI_b$ defined by \eqref{IbCb} is a singleton, and $C_b=\{\underline{b}\}$. 
Therefore, due to the equivalence \eqref{Cbsin} of Theorem~\ref{thm.Cb} we obtain the desired asymptotics \eqref{asyXb­0}.
We could also have concluded directly thanks to the unique ergodicity theorem. However, the unique ergodicity theorem does not apply in the following case, while Theorem~\ref{thm.Cb} does.
\par\medskip\noindent
{\it Case $(ii)$.} Let $\mu\in\cM_p(Y_1)$ be an invariant measure for the flow $T_t$. The equality $\tilde{\mu}\,b=C$ still holds true in $\R$ for some constant $C\in\R$, as well as equality \eqref{psiC}.
Take the function $\psi:=(|b|+\ep)^{-1}$ for $\ep>0$, in equality \eqref{psiC}, and make $\ep$ tend to $0$. Then, using that $b$ is regular and vanishes in $Y_1$ (which implies that $1/b\notin L^1_\sharp(Y_d)$), and applying successively Beppo-Levi's theorem and \eqref{psiC}, it follows that
\beq\label{bmu=a1}
|C|\times \infty=|C|\times\int_{Y_1}{dy\over |b(y)|}=\lim_{\ep\to 0}\left|\,\int_{Y_1}{C\,dy\over |b(y)|+\ep}\,\right|
=\liminf_{\ep\to 0}\left|\,\int_{Y_1}{b(y)\over |b(y)|+\ep}\,d\mu(y)\,\right|\leq 1,
\eeq
which implies that $C=0$.
Finally, taking $\psi=1$ in the equality \eqref{psiC} with $C=0$, leads us to $C_b=\{0\}$, which by virtue of the equivalence \eqref{Cbsin} of Theorem~\ref{thm.Cb} yields the asymptotics of the flow \eqref{asyXb=0}.
\par
As a by-product, using Lebesgue's dominated convergence theorem in the third integral of~\eqref{bmu=a1}, we also get the equality
\beq\label{bmu=a2}
\int_{\{b(y)\neq 0\}}{b(y)\over |b|}\,d\mu(y)=0.
\eeq
\end{proof}
\subsection{The rectifiable case}
\subsubsection{The case with a fixed direction}
This section deals with the case where the vector field $b$ has a fixed direction, namely
\beq\label{baxi}
b(y)=a(y)\,\xi,\quad y\in Y_d,
\eeq
for some nonnegative function $a\in C^1_\sharp(Y_d)$ which may vanish, and for a given vector $\xi\in\R^d$ with $|\xi|=1$.
%\par
%First note that the case of \eqref{baxi} is actually quite general in dimension two. Indeed, assuming the existence of an invariant measure in $C^1_\sharp(Y_2)$ associated with a nonvanishing vector field~$b\in C^1_\sharp(Y_2)^2$, by virtue of Kolmogorov's theorem  any smooth dynamical system \eqref{bX} can be transformed thanks to a smooth diffeomorphism on $Y_2$ into a dynamical system with a vector field of type~\eqref{baxi} (see, {\em e.g.} \cite[Theorem~2.1]{Tas}).
\par
We have the following result.
\begin{apro}\label{pro.asyxi}
Let $b\in C^1_\sharp(Y_d)^d$ be given by \eqref{baxi} with $a\geq 0$ and $\xi\in\R^d$ with $|\xi|=1$.
\begin{enumerate}[$(i)$]
\item If $\xi\cdot k\neq 0$ for any $k\in\Z^d\setminus\{0\}$, then we have
\beq\label{asyXaxi}
\forall\,x\in\R^d,\quad\lim_{t\to\infty} \frac{X(t,x)}{t}=a^*\xi\quad\mbox{with}\quad
a^*:=\left\{\ba{ll}
\dis \underline{a} & \mbox{if $a>0$ in }Y_d
\\ \ecart
\dis 0 & \mbox{if $a$ vanishes in }Y_d,
\ea\right.
\eeq
where $\underline{a}$ is the harmonic mean of $a$.
\item If there exists $T>0$ such that $T\xi\in\Z^d$, then we have
\beq\label{asyXaxiT}
\ba{c}
\dis \forall\,x\in\R^d,\quad \lim_{t\to\infty} \frac{X(t,x)}{t}=a^*(x)\,\xi\quad\mbox{with}
\\ \ecart
\dis a^*(x):=\left\{\ba{cl}
\dis \left({1\over T}\int_0^T\kern -.2cm{ds\over a\big(s\xi+\Pi_{\xi^\perp}(x)\big)}\right)^{-1}\kern -.1cm\xi
& \mbox{if }\ \forall\,u\in\R,\ a\big(u\xi+\Pi_{\xi^\perp}(x)\big)\neq 0
\\ \ecart
0 & \mbox{if }\ \exists\,u\in\R,\ a\big(u\xi+\Pi_{\xi^\perp}(x)\big)=0,
\ea\right.
\ea
\eeq
where $\Pi_{\xi^\perp}$ denotes the projection on the hyperplane $\xi^\perp$ orthogonal to $\xi$.
\end{enumerate}
\end{apro}
\begin{arem}
The two-dimensional framework is completely covered by the disjoint cases $(i)$ and $(ii)$ of Proposition~\ref{pro.asyxi}, since
\[
\exists\,k\in\Z^d\setminus\{0\},\;\;\xi\cdot k=0\;\Leftrightarrow\; \exists\,T>0,\;\;T\xi\in\Z^2.
\]
This equivalence does not hold in higher dimension.
\end{arem}
\begin{proof}{ of Proposition~\ref{pro.asyxi}}
\par\noindent
{\it Case $(i)$}. Assume that for any $k\in\Z^d\setminus\{0\}$, $\xi\cdot k\neq 0$.
\par
First assume that $a$ does not vanish in $Y_d$. Let $x\in\R^d$. Since the vector field $b$ is parallel to the fixed direction $\xi$, we have
\beq\label{Pixio}
\forall\,t\in\R,\quad X(t,x)=\big(X(t,x)\cdot\xi\big)\,\xi+\Pi_{\xi^\perp}\big(X(t,x)\big) \quad\mbox{and}\quad \Pi_{\xi^\perp}\big(X(t,x)\big)=\Pi_{\xi^\perp}(x),
\eeq
which implies that
\beq\label{X.xi}
\left\{\ba{ll}
\dis {\partial (X\cdot\xi)\over\partial t}(t,x)=a\big((X\cdot\xi)(t,x)\,\xi+\Pi_{\xi^\perp}(x)\big), & t\in\R
\\ \ecart
(X\cdot\xi)(0,x)=x\cdot\xi. &
\ea\right.
\eeq
Then, the solution $X\cdot\xi$ to equation \eqref{X.xi} is given by
\beq\label{Fx}
\forall\,t\in\R,\quad X(t,x)\cdot\xi=F_x^{-1}\big(t+F_x(x\cdot\xi)\big)\quad\mbox{where}\quad
F_x(t):=\int_0^t {ds\over a\big(s\,\xi+\Pi_{\xi^\perp}(x)\big)}.
\eeq
By approximating in $C^0_\sharp(Y_d)$ the continuous $\Z^d$-periodic function $1/a$ by Fej\'er's type trigonometric polynomials, and noting that by hypothesis
\[
\forall\,k\in \Z^d\setminus\{0\},\quad \lim_{t\to \infty}\left({1\over t}\int_0^t e^{-2i\pi\,s\,\xi\cdot k}\,ds\right)=0,
\]
it follows that
\[
\lim_{t\to \infty}{F_x(t)\over t}=\int_{Y_d}{dy\over a(y)}={1\over \underline{a}},
\]
which taking into account that $\dis F_x\big(X(t,x)\cdot\xi\big)\mathop{\sim}_{t\to\infty}t$ with $X(t,x)\cdot\xi\to\infty$ as $t\to\infty$, implies that
\[
\lim_{t\r\infty} \frac{X(t,x)\cdot\xi}{t}=\underline{a}.
\]
Therefore, we get that
\[
\forall\,x\in\R^d,\quad \lim_{t\r\infty} \frac{X(t,x)}{t}=\lim_{t\r\infty} {\big(X(t,x)\cdot\xi\big)\,\xi+\Pi_{\xi^\perp}(x)\over t}=\underline{a}\,\xi.
\]
\par
Otherwise, if $a$ vanishes in $Y_d$ let us prove that
\[
\forall\,\mu\in\cI_{a\xi},\quad \int_{Y_d} a(y)\,d\mu(y)=0,
\]
or equivalently, $C_{a\xi}=\{0\}$. Then, Theorem~\ref{thm.Cb} will allow us to conclude.
Assume by contradiction that there exists a measure $\mu\in\cI_{a\xi}$ such that
\[
\int_{Y_d} a(y)\,d\mu(y)>0.
\]
Set $a_\ep:=a+\ep$ for $\ep>0$, and define the probability measure (recall that $a\geq 0$)
\beq\label{muep}
d\mu_\ep(x):=\left(\int_{Y_d}{a(y)\over a_\ep(y)}\,d\mu(y)\right)^{-1}{a(x)\over a_\ep(x)}\,d\mu(x),
\eeq
which is well defined since
\[
\int_{Y_d} {a(y)\over a_\ep(y)}\,d\mu(y)\geq {1\over \|a\|_\infty+\ep}\int_{Y_d} a(y)\,d\mu(y)>0.
\]
Since $\mu$ is invariant for the flow associated with $a\,\xi$, by equality \eqref{dmub=0} we have
\[
\forall\,\ph\in C^1_\sharp(Y_d),\quad \int_{Y_d} a_\ep(x)\,\xi\cdot\nabla\ph(x)\,d\mu_\ep(x)
=\left(\int_{Y_d}{a(y)\over a_\ep(y)}\,d\mu(y)\right)^{-1}\int_{Y_d} a(x)\,\xi\cdot\nabla\ph(x)\,d\mu(x)=0,
\]
hence $\mu_\ep\in\cI_{a_\ep\xi}$. But from the former case $a>0$ combined with Theorem~\ref{thm.Cb} we deduce that
\[
C_{a_\ep\xi}=\left\{\int_{Y_d}a_\ep(y)\,\mu_\ep(dy)\,\xi\right\}=\{\underline{a_\ep}\,\xi\}\quad\mbox{or equivalently}\quad
\int_{Y_d}a_\ep(y)\,\mu_\ep(dy)=\underline{a_\ep}.
\]
This combined with the expression of $a_\ep(x)\,d\mu_\ep(x)$ given by \eqref{muep} leads us to the equality
\[
\left(\int_{Y_d}{a(y)\over a_\ep(y)}\,d\mu(y)\right)\underline{a_\ep}
=\left(\int_{Y_d}{a(y)\over a_\ep(y)}\,d\mu(y)\right)\left(\int_{Y_d}a_\ep(x)\,\mu_\ep(dx)\right)=\int_{Y_d}a(x)\,\mu(dx).
\]
Then, applying Beppo-Levi's theorem and using that $a$ is regular and vanishes in $Y_d$ (which implies that $1/a\notin L^1_\sharp(Y_d)$), it follows that
\[
0<\int_{Y_d}a(x)\,\mu(dx)=\left(\int_{Y_d}{a(y)\over a_\ep(y)}\,d\mu(y)\right)\underline{a_\ep}
\leq \underline{a_\ep}\;\mathop{\longrightarrow}_{\ep\to 0}\;
\left(\int_{Y_d}{dy\over a(y)}\right)^{-1}=0,
\]
which yields a contradiction.
Therefore, we get that $C_{a\xi}=\{0\}$, and by the convergence \eqref{Cbsin} of Theorem~\ref{thm.Cb} we obtain that
\[
\forall\,x\in\R^d,\quad \lim_{t\r\infty} \frac{X(t,x)}{t}=0.
\]
\par\medskip\noindent
{\it Case $(ii)$}.  Now, assume that there exists $T>0$ such that $T\xi\in\Z^d$.
Let $x\in\R^d$. We have to distinguish two cases.
\par
If $a\big(u\xi+\Pi_{\xi^\perp}(x)\big)>0$ for any $u\in\R$, then by the $T$-periodicity of $u\mapsto a\big(u\xi+\Pi_{\xi^\perp}(x)\big)$ we have
\[
\lim_{t\to \infty}{F_x(t)\over t}={1\over T}\int_0^T{ds\over a\big(s\xi+\Pi_{\xi^\perp}(x)\big)},
\]
hence
\[
\lim_{t\r\infty} \frac{X(t,x)}{t}=\left({1\over T}\int_0^T{ds\over a\big(s\xi+\Pi_{\xi^\perp}(x)\big)}\right)^{-1}\kern -.1cm\xi.
\]
\par
On the contrary, if there exists $u_0\in\R$ such that $a\big(u_0\xi+\Pi_{\xi^\perp}(x)\big)=0$, then the part $(ii)$ of Proposition~\ref{pro.1D} applies to the one-dimensional solution $X\cdot\xi$ to \eqref{X.xi}, where $u\mapsto a\big(u\xi+\Pi_{\xi^\perp}(x)\big)$ is $T$-periodic and vanishes at the point $u_0$. We thus deduce that
\[
\lim_{t\to \infty}{X(t,x)\cdot\xi\over t}=0,\quad\mbox{then}\quad \lim_{t\to \infty}{X(t,x)\over t}=0,
\]
which concludes the proof.
\end{proof}
\subsubsection{Rectification to a fixed direction}
In this section we extend the result of the former section thanks to a diffeomorphism on the torus.
\begin{adef}\label{def.diffY}
A mapping $\Phi\in C^1(\R^d)^d$ is said to be a $C^1$-diffeomorphism on the torus if $\Phi$ satisfies the following conditions:
\begin{itemize}
\item $\det(\Phi(x))\neq 0$ for any $x\in\R^d$,
\item there exist a matrix $A\in\Z^{d\times d}$ with $|\det(A)|=1$, and a $\Z^d$-periodic mapping $\Phi_\sharp\in C^1_\sharp(Y_d)^d$ such that
\beq\label{PhAPhd}
\forall\, x\in\R^d, \quad \Phi(x)=Ax+\Phi_\sharp(x).
\eeq
\end{itemize}
\end{adef}
Note that the invertibility of $A$ and the periodicity of $\Phi_\sharp$ in \eqref{PhAPhd} imply that $\Phi$ is a proper function ({\em i.e.}, the inverse image by the function of any compact set in $\R^d$ is a compact set). Hence, by virtue of Hadamard-Caccioppoli's theorem \cite{Cac} (also called Hadamard-L\'evy's theorem) the mapping $\Phi$ is actually a $C^1$-diffeomorphism on $\R^d$. Also noting that due to $A^{-1}\in\Z^{d\times d}$, we have
\[
\forall\,k\in\Z^d,\ \forall\,x\in\R^d,\quad
\left\{\ba{ll}
\Phi(x+k)-\Phi(x)=Ak & \in\Z^d
\\ \ecart
\Phi^{-1}(x+k)-\Phi^{-1}(x)=A^{-1}k & \in\Z^d,
\ea\right.
\]
$\Phi$ well defines an isomorphism on the torus.
\par
Using a diffeomorphism on the torus the result of Proposition~\ref{pro.asyxi} can be extended to the following general result.
\begin{acor}\label{cor.bPhaxi}
Consider a vector field $b$ such that there exist a $C^1$-diffeomorphism $\Phi$ on the torus given by \eqref{PhAPhd}, a nonnegative function $a\in C^1_\sharp(Y_d)^d$ and a vector $\xi\in\R^d$ with $|\xi|=1$, such that
\beq\label{bPhaxi}
\forall\,x\in\R^d,\quad b(x)=a(\Phi(x))\,\nabla\Phi(x)^{-1}\xi.
\eeq
Then, if $\xi$ satisfies one of the two conditions of Proposition~\ref{pro.asyxi} we get that
\beq\label{asyXa*Axi}
\forall\,x\in\R^d,\quad\lim_{t\to\infty} \frac{X(t,x)}{t}=a^*(\Phi(x))\,A^{-1}\xi,
\eeq
where $a^*(y)$ is given by formula~\eqref{asyXaxi} or formula~\eqref{asyXaxiT}.
\end{acor}
\begin{arem}
A similar result as Corollary~\ref{cor.bPhaxi} was proved by Tassa \cite[Theorem~2.1]{Tas} in dimension two assuming that the first coordinate $b_1$ of $b$ does not vanish in $\R^2$, and that there exists an invariant measure for the flow associated with $b$, having a density in $C^1_\sharp(Y_2)$ with respect to Lebesgue's mesure. More precisely, under these two conditions Tassa showed in the spirit of Kolmogorov's theorem the existence of a $C^1$-diffeomorphism on the torus satisfying \eqref{bPhaxi} with the matrix $A=I_2$. In the present case, we only assume the weaker rectification formula \eqref{bPhaxi} in any dimension $d\geq 1$, with any diffeomorphism $\Phi$ on the torus and a possibly vanishing nonnegative function $a$.
\par
Conversely, consider a vector field $b\in C^1_\sharp(Y_d)^d$ satisfying \eqref{bPhaxi} with $a>0$. By making the change of variables $\ph(x)=\psi(y)$ with $y=\Phi(x)$ for any function $\ph\in C^\infty_c(\R^d)$, it is easy to check that
\[
\int_{\R^d}{\det(\nabla\Phi(x))\over a(\Phi(x))}\,b(x)\cdot\nabla\ph(x)\,dx=\int_{\R^d}\xi\cdot\nabla\psi(y)\,dy=0,
\]
which implies that
\beq\label{siaPhi}
\div(\sigma b)=0\;\;\mbox{in }\R^d\quad\mbox{with}\quad \sigma:=\left(\int_{Y_d}{\det(\nabla\Phi(y))\over a(\Phi(y))}\,dy\right)^{-1}{\det(\nabla\Phi)\over a\circ \Phi}>0.
\eeq
Since the function $\sigma$ is $\Z^d$-periodic, by virtue of Theorem~\ref{thm.divcurl} \eqref{siaPhi} means that the density probability measure $\sigma(x)\,dx$ is invariant for the flow associated with $b$.
\par
Therefore, the rectification formula \eqref{bPhaxi} of a nonvanishing vector field $b$ can be regarded as a (more restrictive) substitute to \eqref{siaPhi} in dimension $d>2$ for which Kolmogorov's theorem does not apply.
\end{arem}
\begin{proof}{ of Corollary~\ref{cor.bPhaxi}}
Define
\beq\label{YPhX}
Y(t,y):=\Phi(X(t,x))\;\;\mbox{for }(t,x)\in\R\times\R^d\quad\mbox{with}\quad y:=\Phi(x).
\eeq
By the chain rule and \eqref{bPhaxi} we have
\[
\left\{\ba{l}
\dis {\partial Y\over\partial t}(t,y)=\nabla\Phi(X(t,x))\,{\partial X\over\partial t}(t,x)=\nabla\Phi(X(t,x))\,b(X(t,x))=a(\Phi(X(t,x)))\,\xi=a(Y(t,y))\,\xi
\\ \ecart
\dis Y(0,y)=\Phi(X(0,x))=\Phi(x)=y,
\ea\right.
\]
which shows that $Y$ is the solution to the dynamical system associated with the vector field~$a\,\xi$.
Hence, if $\xi$ satisfies one of the two conditions of Proposition~\ref{pro.asyxi}, we get the asymptotics of the flow
\[
\forall\,y\in\R^d,\quad\lim_{t\to\infty} \frac{Y(t,y)}{t}=a^*(y)\,\xi,
\]
where $a^*(y)$ is given by formula~\eqref{asyXaxi} or formula~\eqref{asyXaxiT}.
Therefore, by the boundedness of $\Phi_\sharp$ in~\eqref{PhAPhd} we have for any $x\in\R^d$ and $y:=\Phi(x)$,
\[
\frac{X(t,x)}{t}=\frac{\Phi^{-1}(Y(t,y))}{t}=A^{-1}\left[\frac{A\big(\Phi^{-1}(Y(t,y))\big)}{t}\right]=A^{-1}\left[\frac{Y(t,y)}{t}\right]+{O(1)\over t}
\;\mathop{\longrightarrow}_{t\to\infty}\;a^*(y)\,A^{-1}\xi,
\]
which yields the desired asymptotics \eqref{asyXa*Axi}.
\end{proof}
\begin{arem}\label{rem.P}
Peirone~\cite[Theorem~3.1]{Pei} proved that in dimension two and for a nonvanishing vector field $b\in C^1_\sharp(Y_2)^2$, there is no periodic solution in $\R^2$ to system~\eqref{bX} (see the definition after \eqref{XtperY}), and that the limit of $X(t,x)/t$ as $t\to\infty$ does exist for any $x\in\R^2$ with the following alternative:
\begin{itemize}
\item If system~\eqref{bX} has no periodic solution in the torus according to \eqref{XtperY}, then the limit of $X(t,x)/t$ as $t\to\infty$ is independent of $x\in\R^2$ as in \eqref{asyXaxi}.
\item If system~\eqref{bX} has a periodic solution in the torus, then the limit of $X(t,x)/t$ as $t\to\infty$ may depend on $x\in\R^2$ as in \eqref{asyXaxiT}.
\end{itemize}
Actually, in any dimension the alternative on the vector $\xi$ in Corollary~\ref{cor.bPhaxi} (see also Proposition~\ref{pro.asyxi}) is naturally connected to the former alternative due to the following result.
\end{arem}
\begin{apro}\label{pro.P}
In the framework of Corollary~\ref{cor.bPhaxi} assume in addition that
\beq\label{ax>0}
\exists\,x\in\R^d,\ \forall\,t\in\R,\quad a\big(t\,\xi+\Pi_{\xi^\perp}(x)\big)>0.
\eeq
Then, the following equivalence holds:
\beq\label{Xper}
\exists\,T>0,\;\;T\xi\in\Z^d\;\Leftrightarrow\;\mbox{system \eqref{bX} has a periodic solution in $Y_d$ but not in $\R^d$}.
\eeq
Moreover, the implication $(\Leftarrow)$ of \eqref{Xper} is always true. When condition \eqref{ax>0} is not satisfied, the implication $(\Rightarrow)$ of\eqref{Xper} does not hold in general.
\end{apro}
\begin{proof}{}
First note that the definition \eqref{PhAPhd} of the $C^1$-diffeomorphism on the torus $\Phi$ implies that a solution $X(\cdot,x)$ is periodic in the torus if, and only if, the function $Y(\cdot,\Phi(x))$ is periodic in the torus. Therefore, it is enough to prove the result in the case $b = a\,\xi$.
\par
If system \eqref{bX} with $b=a\,\xi$ has a periodic solution $X(\cdot,x)$ in the torus but not in $\R^d$, then there exists $\tau>0$ and $k\in\Z^d\setminus\{0\}$ such that by \eqref{Pixio} we have for any $t\in\R$,
\[
\Pi_{\xi^\perp}(x)=\Pi_{\xi^\perp}(X(t+\tau,x))=\Pi_{\xi^\perp}(X(t,x)+k)=\Pi_{\xi^\perp}(X(t,x))+\Pi_{\xi^\perp}(k)=\Pi_{\xi^\perp}(x)+\Pi_{\xi^\perp}(k).
\]
Hence, we obtain that $\Pi_{\xi^\perp}(k)=0$, or equivalently, there exists $T>0$ such that $T\xi=\pm k\in\Z^d$.
\par
Conversely, assume that \eqref{ax>0} is satisfied for some $x\in\R^d$, and there exists $T>0$ such that $T\xi=k\in\Z^d$.
The solution $X(\cdot,x)$ of \eqref{bX} is given by \eqref{Pixio} and \eqref{Fx}. Define
\[
\tau:=\int_0^T {ds\over a\big(s\,\xi+\Pi_{\xi^\perp}(x)\big)}.
\]
Then, by the $T$-periodicity of the function $t\mapsto a\big(t\,\xi+\Pi_{\xi^\perp}(x)\big)$ we have
\[
\ba{ll}
\dis F_x(x\cdot\xi+T) & \dis =\int_0^{x\cdot\xi+T}{ds\over a\big(s\,\xi+\Pi_{\xi^\perp}(x)\big)}
\\ \ecart
& \dis =\int_0^{x\cdot\xi}{ds\over a\big(s\,\xi+\Pi_{\xi^\perp}(x)\big)}+\int_{x\cdot\xi}^{x\cdot\xi+T}{ds\over a\big(s\,\xi+\Pi_{\xi^\perp}(x)\big)}
\\ \ecart
& \dis =F_x(x\cdot\xi)+\tau=F_x\big(X(\tau,x)\cdot\xi\big)\quad \mbox{by \eqref{Fx}},
\ea
\]
which implies that $X(\tau,x)\cdot\xi=x\cdot\xi+T$.
Moreover, the functions $t\mapsto X(t+\tau,x)\cdot\xi$ and $t\mapsto X(t,x)\cdot\xi+T$ are solutions to the system $z'(t)=a\big(z(t)\,\xi+\Pi_{\xi^\perp}(x)\big)$, and agree at $t=0$. Hence, by a uniqueness argument these two solutions are equal, which yields
\[
\forall\,t\in\R,\quad X(t+\tau,x)\cdot\xi=X(t,x)\cdot\xi+T=\big(X(t,x)+k\big)\cdot\xi.
\]
Moreover, again by \eqref{Pixio} and recalling that $\Pi_{\xi^\perp}(k)=0$  we have
\[
\forall\,t\in\R,\quad \Pi_{\xi^\perp}\big(X(t+\tau,x)\big)=\Pi_{\xi^\perp}(x)=\Pi_{\xi^\perp}\big(X(t,x)\big)=\Pi_{\xi^\perp}\big(X(t,x)+k\big).
\]
The two previous identities imply that
\[
\forall\,t\in\R,\quad X(t+\tau,x)=X(t,x)+k.
\]
namely, $X(\cdot,x)$ is periodic in the torus.
\par
Consider the vector field $b(x) := a(x) e_1$ with $a(x):={1\over\pi}\cos^2(\pi x_1)\geq 0$ for $x\in\R^d$. Here, the vector $\xi=e_1$ clearly satisfies the left-hand side of \eqref{Xper}. However, since
\[
\forall\,x\in\R^d,\quad \textstyle a\big({1\over2}e_1+\Pi_{e_1^\perp}(x)\big)={1\over\pi}\cos^2({\pi\over2})=0,
\]
condition \eqref{ax>0} is not satisfied. Moreover, for any $x\in\R^d$ the solution $X(\cdot,x)$ to \eqref{bX} is given by
\[
\left\{\ba{l}
X_1(t,x)=\left\{\ba{cl}
x_1 & \mbox{if }x_1={1\over2}+n,\ n\in\Z
\\ \ecart
{1\over\pi}\arctan(t+\tan(\pi x_1))+n & \mbox{if }x_1\in(-{1\over2}+n,{1\over2}+n),\ n\in\Z
\ea\right.
\\ \ecart
X_i(t,x)=x_i\ \mbox{for }i\geq 2
\ea\right.
\quad (t,x)\in\R\times\R^d.
\]
It follows that any periodic solution in the torus is necessarily stationary and {\em a fortiori} periodic in $\R^d$ in the sense of \eqref{XtperY}. Therefore, the implication $(\Rightarrow)$ of \eqref{Xper} does not hold.
\end{proof}
\subsection{The case of a current field}
The following result deals with the case of a current field $b=A\nabla v$ defined with some matrix-valued conductivity $A$ and some electric field $\nabla v$ with zero average.
\begin{apro}\label{pro.ADv}
Assume that
\beq\label{bDu}
b=A\nabla v\;\;\mbox{in }Y_d\quad\mbox{with}\quad
\left\{\ba{ll}
\dis A\in C^1_\sharp(Y_d)^{d\times d}, & \dis A=A^T\geq 0\mbox{ in }Y_d,
\\ \ecart
\dis v\in C^2_\sharp(Y_d)^d. &
\ea\right.
\eeq
Then, the flow $T_t$ associated with $b$ satisfies the asymptotics
\[
\forall\,x\in\R^d,\quad \lim_{t\to\infty}{X(t,x)\over t}=0.
\]
\end{apro}
\begin{proof}{}
Let $\mu\in\cM_p(Y_d)$ be an invariant measure for the flow $T_t$.
By the div-curl relation \eqref{divcurl} combined with the zero average of $\nabla v$, we have
\[
\int_{Y_d}\underbrace{A(y)\nabla v(y)\cdot\nabla v(y)}_{\geq 0}\,\mu(dy)=\int_{Y_d}b(y)\cdot\nabla v(y)\,\mu(dy)=0,
\]
which implies that $A\nabla v\cdot\nabla v=0$ $\mu$-a.e. in $Y_d$. Since the matrix-valued $A$ is symmetric and nonnegative, by the Cauchy-Schwarz inequality we deduce that $A\nabla v=0$ $\mu$-a.e. in $Y_d$, and thus
\[
\int_{Y_d}b(y)\,\mu(dy)=\int_{Y_d}A(y)\nabla v(y)\,\mu(dy)=0.
\]
Therefore, $C_b=\{0\}$ and the equivalence \eqref{Cbsin} of Theorem~\ref{thm.Cb} allows us to conclude.
\end{proof}

%%%%%%%%%%
\appendix 

\section{Proof of Lemma~\ref{lem-extrac}}
%\begin{proof}{}
Since $Y_d$ is a compact metrizable space, there exists a subsequence $(\nu_{n_k})_{k\in\N}$ of $(\nu_n)_{n\in\N}$ which converges weakly~$*$ to some probability measure $\mu\in \cM_p(Y_d)$, namely for any $f\in C^0_\sharp(Y_d)$,
\beq\label{wcnunk}
\int_{Y_d} f(y)\, d\nu_{n_k}(y) = \frac{1}{r_{n_k}} \int_0^{r_{n_k}} f(X(s,x))\,ds\;\mathop{\longrightarrow}_{k\to\infty}\;\int_{Y_d} f(y)\, d\mu(y).
\eeq
Let us prove that $\mu$ is invariant for the flow $T_t$. For the sake of simplicity denote $\tau_k:= r_{n_k}$ and $\mu_k:= \nu_{n_k}$.
Let $t\in\R$ and $f\in C^0_\sharp(Y_d)$.
By the semi-group property of the flow \eqref{sgroup} we have
\[
\int_{Y_d} (T_t f)(y)\, d\mu_k(y) =  \frac{1}{\tau_k} \int_0^{\tau_k} f(X(s+t,x))\,ds.
\]
By the change of variable $r=s+t$, it follows that 
\[
\ba{l}
\dis \int_{Y_d} (T_t f)(y)\, d\mu_k(y) =\frac{1}{\tau_k} \int_{t}^{t+\tau_k} f(X(r,x))\,dr
\\ \ecart
\dis =\frac{1}{\tau_k} \int_{0}^{\tau_k} f(X(r,x))\,dr + \frac{1}{\tau_k} \int_{\tau_k}^{t+\tau_k} f(X(r,x))\,dr- \frac{1}{\tau_k} \int_{0}^{t} f(X(r,x))\,dr
\ea
\]
Since $f$ is bounded and $t\in\R$ is fixed, we deduce from \eqref{wcnunk} that
\[
\lim_{k\to \infty} \int_{Y_d} (T_t f)(y)\, d\mu_k(y) = \int_{Y_d} f(y)\, d\mu(y).
\]
However, by the definition of $\mu$ we also have
\[
\lim_{k\to \infty} \int_{Y_d} (T_t f)(y)\, d\mu_k(y) = \int_{Y_d} (T_t f)(y)\, d\mu(y).
\]
Hence, we get that
\[
\forall\,t\in\R,\ \forall\,f\in C^0_\sharp(Y_d),\quad \int_{Y_d} (T_t f)(y)\, d\mu(y) = \int_{Y_d} f(y)\, d\mu(y),
\]
which implies that $\mu$ is invariant for the flow $T_t$. 
\fdem
%\end{proof}
%%%%%%%%%%

%%%%%%%%%%
\end{document}